\documentclass{birkjour}
\usepackage{amsmath,amssymb,slashed}
 \newtheorem{thm}{Theorem}[section]
 
 \newtheorem{lem}[thm]{Lemma}
 
 \theoremstyle{definition}
 \newtheorem{defn}[thm]{Definition}
 \theoremstyle{remark}

 \numberwithin{equation}{section}
      \renewcommand{\epsilon}{\varepsilon}
\newcommand{\Z}{\mathbb{Z}}
\newcommand{\ux}{\underline{x}}
\newcommand{\uy}{\underline{y}}
\newcommand{\uz}{\underline{z}}

\newcommand{\mZ}{\mathbb{Z}}

\newcommand{\mR}{\mathbb{R}}

\newcommand{\mC}{\mathbb{C}}

\newcommand{\mS}{\mathbb{S}}
\usepackage{enumerate}

\usepackage{hyperref}

\newcommand{\mcH}{\mathcal{H}}

\newcommand{\mcP}{\mathcal{P}}

\newcommand{\uv}{\underline{v}}
\newcommand{\uu}{\underline{u}}

\newcommand{\up}{\underline{\partial}}
\newcommand{\mE}{\mathbb{E}}

\usepackage[framemethod=tikz]{mdframed}

\newcommand{\R}{\mathbb{R}}

\newcommand{\spa}{\operatorname{span}}

\newcommand{\C}{\mathbb{C}}

\begin{document}

%
%
%
%
%
%
%
%
%

\title[Branching symplectic monogenics using a M--Z  algebra]
 {Branching symplectic monogenics using \\ a Mickelsson--Zhelobenko algebra}

\author[David Eelbode]{David Eelbode}

\address{%
	Department of Mathematics \newline
University of Antwerp \newline
Middelheimlaan 1 \newline 
2020 Antwerp, Belgium}

\email{david.eelbode@uantwerpen.be}

\author{Guner Muarem}
\address{%
	Department of Mathematics \newline
	University of Antwerp \newline
	Middelheimlaan 1 \newline 
	2020 Antwerp, Belgium}
\email{guner.muarem@uantwerpen.be}


\date{January 1, 2004}
\dedicatory{}

\begin{abstract}
In this paper we consider (polynomial) solution spaces for the symplectic Dirac operator (with a focus on $1$-homogeneous solutions). This space forms an infinite-dimensional representation space for the symplectic Lie algebra $\mathfrak{sp}(2m)$. Because $\mathfrak{so}(m)\subset \mathfrak{sp}(2m)$, this leads to a branching problem which generalises the classical Fischer decomposition in harmonic analysis. Due to the infinite nature of the solution spaces for the symplectic Dirac operators, this is a non-trivial question: both the summands appearing in the decomposition and their explicit embedding factors will be determined in terms of a suitable Mickelsson-Zhelobenko algebra. 
\end{abstract}

\maketitle
\section{Introduction}
\noindent
The Dirac operator is a first-order differential operator acting on spinor-valued functions which factorises the Laplace operator $\Delta$ on $\mathbb{R}^m$. It was originally introduced by Dirac in a famous attempt to factorise the wave operator, hence obtaining a relativistically invariant version of the Schr\"odinger equation.  Since then, this operator has played a crucial role in mathematical domains such as representation theory and Clifford analysis. The latter is a multidimensional function theory which is often described as a refinement of harmonic analysis, and a generalisation of complex analysis. It is centred around a generalisation of the operator introduced by Dirac (his operator $\slashed{\partial}$ is defined in 4 dimensions), and can be seen as a contraction between the generators $e_k$ for a Clifford algebra (acting as endomorphisms on so-called spinors) and corresponding partial derivatives $\partial_{x_k}$. To be more precise, introducing the Clifford algebra by means of the defining relations $\{e_a,e_b\} = e_ae_b + e_be_a =-2\delta_{ab}$ (with $1 \leq a,b \leq m$) the Dirac operator is given by
\[ \up_x = \begin{pmatrix}
	e_1& \ldots &e_m
\end{pmatrix}\textup{Id}_m\begin{pmatrix}
	x_1\\
\vdots\\
x_m
\end{pmatrix} = \sum_{j = 1}^m e_j \partial_{x_j}\ , \]
whereby the $(m\times m)-$identity matrix $\textup{Id}_m$ has been added to explain what is meant by the `contraction'.  Null-solutions for $\up_x$ are called monogenics, and can be seen as generalisations of holomorphic functions. One often starts with the study of $k$-homogeneous polynomial solutions for the Dirac operator, which belong to the space $\mathcal{M}_k(\mathbb{R}^m,\mathbb{S})$, where $\mathbb{S}$ stands for the aforementioned spinor space. \par 
An obvious generalisation of the operator $\up_x$ can be obtained by using another matrix than $\textup{Id}_m$ when contracting algebraic generators with partial derivatives. An important example is the {\em symplectic} Dirac operator, which is introduced on a symplectic space rather than an orthogonal space (see for example the work of Habermann \cite{habermann}). This operator, denoted by $D_s$, is defined as a contraction between generators for a symplectic Clifford algebra and partial derivatives, using a skew-symmetric matrix $\Omega_0$ (rather than $\mbox{Id}_m$). The symplectic Clifford algebra generators satisfy the Heisenberg relations $[\partial_{z_j},z_k]=\delta_{jk}$ (the symplectic analogue of the Clifford relations for the generators $e_k$ from above). Note that the symbols $z_j$ stand for real variables here, they are chosen because the sets of (real) variables $x_j$ and $y_j$ will also appear in this paper. In sharp contrast to the orthogonal case, the symplectic Clifford algebra is no longer finite-dimensional. This trend continues, in the sense that the associated symplectic spinor space $\mathbb{S}^{\infty}_0$ also becomes infinite-dimensional. 
\par
In this paper, we study infinite-dimensional spaces defined in terms of solutions for the symplectic Dirac operator (generalised monogenics). These spaces can be defined algebraically
\[ \mathbb{S}_k^{\infty} = \mathcal{M}_k^s(\mathbb{R}^{2m},\mathbb{S}^{\infty}_0):=\mathcal{P}_k(\mathbb{R}^{2m},\mathbb{C})\boxtimes \mathbb{S}^{\infty}_0\ \quad (k \in \mathbb{N}). \]
Here $\boxtimes$ denotes the Cartan product of the $\mathfrak{sp}(2m)$-representations $\mathcal{P}_k(\mathbb{R}^{2m},\mathbb{C})$, the $k$th-symmetric power of the fundamental vector representation (modelled by polynomials), and the symplectic spinor space $\mathbb{S}^{\infty}_0$ (also referred to as the Segal-Shale-Weil representation). These spaces contain $k$-homogeneous $\mathbb{S}^\infty_0$-valued solutions for the symplectic Dirac operator. The behaviour of these spaces as representations for $\mathfrak{sp}(2m)$ is known (see e.g. \cite{BDS} and the references therein), but in this paper we will look at these spaces as {\em orthogonal} representation spaces. This is motivated by the fact that $\mathfrak{so}(m) \subset \mathfrak{sp}(2m)$, which means that we are dealing with a branching problem. \par In general, a branching problem can be described as follows: given a representation $\rho$ of a Lie algebra $\mathfrak{g}$ and a subalgebra $\mathfrak{h}$, we would like to understand how the representation $\rho$ behaves as a $\mathfrak{h}$-representation. This restricted representation $\rho_{|\mathfrak{h}}$ will no longer be irreducible, but will decompose into $\mathfrak{h}$-irreducible representations. A branching rule then describes the irreducible pieces which will occur, together with their multiplicities. For the symplectic spinors (i.e. for the space $\mathbb{S}^\infty_0$), this gives the Fischer decomposition in harmonic analysis, which means that the branching problem for $\mathbb{S}^\infty_k$ leads to generalisations thereof. To describe the branching of the infinite-dimensional symplectic representation space $\mathbb{S}_k^{\infty}$ under the inclusion $\mathfrak{so}(m)\subset\mathfrak{sp}(2m)$, we will make use of a quadratic algebra which is known as a Mickelson-Zhelobenko algebra (see \cite{zhelo} for the general construction and properties). 

\section{The symplectic Dirac operator and monogenics}
\noindent We will work with the symplectic space $\R^{2m}$ and coordinates $(\ux,\uy)$ equipped with the canonical symplectic form $\omega_0=\sum_{j=1}^m dx_j\wedge dy_j$. The matrix representation of the symplectic form is given by \begin{align*}
\Omega_0=\begin{pmatrix}
0	& \textup{Id}_m\\
-\textup{Id}_m &	0
\end{pmatrix}.	
\end{align*}
The group consisting of all invertible linear transformations preserving this non-degenerate skew-symmetric bilinear form
is called the symplectic group and is formally defined as follows: $$
\mathsf{Sp}(2m,\R)=\{M\in \mathsf{GL}(2m,\R)\mid M^T\Omega_0M=\Omega_0\}.$$
This is a non-compact group of dimension $2m^2+m$. Its (real) Lie algebra will be denoted by $\mathfrak{sp}(2m,\R)$. In the orthogonal case, the spin group determined by the sequence $$1 \to \mathbb{Z}_2 \to \mathsf{Spin}(m) \to \mathsf{SO}(m) \to 1$$
plays a crucial role concerning the invariance of the Dirac operator $\up_x$ and the definition of the spinors $\mathbb{S}$.
In the symplectic case, this role is played by the metaplectic group $\mathsf{Mp}(2m,\R)$ fixed by the exact sequence $$1 \to \mathbb{Z}_2 \to \mathsf{Mp}(2m,\R) \to \mathsf{Sp}(2m,\R) \to 1.$$ Despite the analogies, there are some fundamental differences:
\begin{enumerate}[(i)]
	\item First of all, the group $\mathsf{SO}(m)$ is compact, whereas $\mathsf{Sp}(2m,\mathbb{R})$ is not. This has important consequences for the representation theory. As a matter of fact, the metaplectic group is not a matrix group and does not admit (faithful) finite-dimensional representations.
	\item  The orthogonal spinors $\mathbb{S}$ can be realised as a maximal left ideal in the Clifford algebra, but this is not the case for the symplectic spinors. The latter are often modelled as smooth vectors in the infinite-dimensional Segal-Shale-Weil representation (see \cite{mpc} and the references therein). One can also identify the symplectic spinor space $\mathbb{S}^{\infty}_0$ with the space $\mathcal{P}(\mathbb{R}^m,\mathbb{C})$ of polynomials in the variables $(z_1,\dots,z_m)\in\mathbb{R}^m$, which is the approach we will use in this paper.
\end{enumerate}
\begin{defn} Let $(V,\omega)$ be a symplectic vector space.
	The \textit{symplectic Clifford algebra} $\mathsf{Cl}_s(V,\omega)$ is defined as the quotient algebra of the tensor algebra $T(V)$ of $V$ by the two-sided ideal $\mathcal{I}_{\omega}:=\{\uv\otimes \uu - \uu\otimes \uv + \omega(\uv,\uu) : \uu,\uv\in V\}$. In other words
	$\mathsf{Cl}_s(V,\omega):=T(V)/\mathcal{I}_{\omega}$
	is the algebra generated by $V$ in terms of the relation $[\uv,\uu]=-\omega(\uv,\uu)$, where we have omitted the tensor product symbols.
\end{defn}
\begin{defn}
	Denote by $\langle \underline{u},\underline{v}\rangle:=\sum_{k=1}^m u_kv_k$ the canonical inner product on $\mathbb{R}^m$ (where we allow partial derivatives to appear as coefficients, see the operators below). We then define the following operators acting on polynomial functions in $\mathcal{P}(\mathbb{R}^{3m},\mathbb{C})$:
	\begin{enumerate}[(i)]
		\item The {symplectic Dirac operator} ${D}_s=\langle \underline{z},\underline{\partial}_y\rangle - \langle \underline{\partial}_x,\underline{\partial}_z \rangle$.
		\item The {adjoint} operator $X_s=\langle \underline{y},\underline{\partial}_z\rangle + \langle \underline{x},\underline{z}\rangle$ with respect to the symplectic Fischer product (see Section 5 of \cite{symplectic} for more details).
		\item The Euler operator $\mathbb{E} = \sum_{j=1}^m (x_j\partial_{x_j} + y_j\partial_{y_j}) = \mathbb{E}_x+\mathbb{E}_y$ measuring the degree of homogeneity in the base variables $(\ux,\uy) \in \mathbb{R}^{2m}$.
	\end{enumerate} 
\end{defn}

\noindent	Note that some authors use the notation $\langle \nabla_x,\nabla_y \rangle$ for an expression such as $\sum_k \partial_{x_k} \partial_{y_k}$, but we will use the Dirac operator symbol here instead of the nabla operator. 

\begin{lem}\label{3operators}
	The three operators $X = \sqrt{2}D_s$, $Y =\sqrt{2}X_s$ and their commutator $H=[X,Y]=-2(\mathbb{E}_x+\mathbb{E}_y+m)$ give rise to a copy of the Lie algebra $\mathfrak{sl}(2)$. 
\end{lem}
\noindent One now easily sees that the symplectic Dirac operator is nothing more than the contraction between the Weyl algebra generators $(z_k,\partial_{z_k})$ with the vector fields $(\partial_{x_k},\partial_{y_k})$ for $k=1,\dots,m$ using the canonical symplectic form $\Omega_0$. 
\begin{defn}
	The space of $k$-homogeneous symplectic monogenics is defined by $\mathbb{S}_k^{\infty}:=\ker(D_s)\cap\left( \mathcal{P}_k(\mathbb{R}^{2m},\mathbb{C})\otimes\mathcal{P}(\mathbb{R}^m,\mathbb{C})\right)$, where the space $\mathcal{P}(\mathbb{R}^m,\mathbb{C})$ in the vector variable $\underline{z} \in \mathbb{R}^m$ plays the role of the symplectic spinor space $\mathbb{S}^\infty_0$.
\end{defn}
\noindent
Note that as an $\mathfrak{sp}(2m,\R)$-module, $\mathbb{S}_k^{\infty}$ is reducible and decomposes into two irreducible parts: $\mathbb{S}_k^{\infty} = \mathbb{S}_{k,+}^{\infty}  \oplus \mathbb{S}_{k,-}^{\infty}$ with highest weights
\[
\mathbb{S}_{k,+}^{\infty}  \longleftrightarrow  \left(	k-\frac{1}{2},-\frac{1}{2},\dots,-\frac{1}{2}	\right) \quad \mbox{and} \quad
\mathbb{S}_{k,+}^{\infty}  \longleftrightarrow  \left(	k-\frac{1}{2},-\frac{1}{2},\dots,-\frac{3}{2}	\right).\]
These weight entries are fixed by the Cartan algebra $\mathfrak{h} = \mathsf{Alg}(X_{jj} : 1 \leq j \leq m)$, where the elements $X_{jj}$ are defined in the lemma below. In this paper, we will omit the parity signs and work with $\mathbb{S}^\infty_k$ as a notation which incorporates both the positive and negative spinors (in our model, this will correspond to even or odd in the variable $\uz \in \mathbb{R}^m$, see below, so it is always easy to `decompose' into irreducible components when necessary).\\ 
\\	
The three operators from Lemma \ref{3operators} can be proven to be invariant under the action of the symplectic Lie algebra, in the sense that they commute with the following generators (see also Lemma 3.3 in \cite{DM}): 
\begin{lem}
	The symplectic Lie algebra $\mathfrak{sp}(2m)$ has the following realisation on the space of symplectic spinor-valued polynomials $\mathcal{P}(\mathbb{R}^{2m},\mathbb{C})\otimes\mathbb{S}^{\infty}_0$:
	\begin{align}\label{realisaties}
	\begin{cases}
	X_{jk}=x_j\partial_{x_k}-y_k\partial_{y_j} - (z_k\partial_{z_j}+\frac{1}{2}\delta_{jk})
	&1 \leq j,k \leq m
	\\Y_{jk}=x_j\partial_{y_k}+x_k\partial_{y_j} -\partial_{z_j}\partial_{z_k}
	&1 \leq j < k \leq m
	\\ Z_{jk}=y_j\partial_{x_k}+y_k\partial_{x_j} +  z_jz_k
	&1 \leq j < k \leq m
	\\Y_{jj}=x_j \partial_{y_j} -\frac{1}{2}\partial_{z_j}^2
	&1 \leq j \leq m
	\\Z_{jj}=y_j\partial_{x_j}+\frac{1}{2} z_j^2
	&1 \leq j\leq m
	\end{cases}
	\end{align} 
\end{lem}
\noindent The branching rule for $\mathbb{S}^\infty_0$, when considering it as a representation space for the orthogonal Lie algebra $\mathfrak{so(m)} \subset \mathfrak{sp}(2m)$, leads to the Fischer decomposition for $\mathbb{C}$-valued polynomials in the variable $\underline{z} \in \mathbb{R}^m$ (see below). Note that $\mathfrak{so}(m)$ is generated by the operators $X_{jk} - X_{kj}$ for $1 \leq j < k \leq m$, giving rise to the well-known angular operators ubiquitous in quantum mechanics (often denoted by $L_{ab}$ with $1 \leq a < b \leq m$). In our previous paper \cite{DM}, we therefore tackled the next case $k=1$ as this is a natural generalisation of said Fischer decomposition.
\noindent
The main problem with our branching rule (Theorem 5.6 in \cite{DM}) is the fact that these $\mathfrak{so}(m)$-spaces appear with infinite multiplicities, which are not always easy to keep track of. Therefore the main goal of this paper is to show that one can organise these in an algebraic framework which extends to other values for $k$ too, using a certain quadratic algebra. 
\section{Simplicial harmonics in three vector variables}
\noindent In this section we describe a generalisation of harmonic polynomials, in three vector variables. This will be done in terms of a solution space for a `natural' collection of $\mathfrak{so}(m)$-invariant differential operators. The corresponding Howe dual pair will be useful for the branching problem addressed above. For the sake of completeness, we recall the following basic definition:
\begin{defn}
	A function $f(\ux)$ on $\R^m$ is called \textit{harmonic} if $\Delta f(\ux)=0$. The \textit{k-homogeneous harmonics} are defined as
	$\mathcal{H}_k(\R^m,\mathbb{C}):=\mathcal{P}_k(\R^m,\mathbb{C})\cap \ker(\Delta)$. These spaces define irreducible representations for $\mathfrak{so}(m)$ with highest weight $(k,0,\dots,0)$ for all $k \in \mathbb{Z}^+$. 
\end{defn}
\noindent It is well-known that the space of $k$-homogeneous polynomials $\mathcal{P}_k(\mathbb{R}^m,\mathbb{C})$ is reducible as an $\mathfrak{so}(m)$-module (see for example \cite{Gilbert}) and decomposes into harmonic polynomials. In fact, the decomposition of the {\em full} space of polynomials is known as the aforementioned \textit{Fischer decomposition}, given by
\[ \mathcal{P}(\mathbb{R}^m,\mathbb{C}) = \bigoplus_{k=0}^{\infty}\mcP_k(\mathbb{R}^m,\mathbb{C}) = \bigoplus_{k=0}^{\infty}\bigoplus_{p=0}^{\infty} |\uz|^{2p}\mathcal{H}_k(\mathbb{R}^m,\mathbb{C}). \]
This can all be generalised to the case of several vector variables (sometimes also called `a matrix variable'): for any highest weight for $\mathfrak{so}(m)$ there is a (polynomial) model in terms of simplicial harmonics (or monogenics for the half-integer representations). We refer to \cite{VSC} for more details. In this paper, we will consider these spaces for $\mathfrak{so}(m)$-weights characterised by three integers  $(a,b,c)$ where $a\geq b\geq c \geq 0$. Also note that trailing zeros in the weight notation will be omitted from now on, so for instance $(k,0,\ldots,0)$ will be written as $(k)$.  First of all, we consider homogeneous polynomials $P_{a,b,c}(\underline{z};\underline{x},\underline{y})$ in three vector variables $(\underline{z};\underline{x},\underline{y})\in\mathbb{R}^{3m}$. Here we use the notation $(\underline{z};\underline{x},\underline{y})$ to stress the difference between the variable $\underline{z}$ (the spinor variable, referring to an element in $\mathbb{S}^\infty_0$) from the other two variables $(\underline{x},\underline{y}) \in \mathbb{R}^{2m}$, which are `ordinary' variables. The parameters $(a,b,c)$ then refer to the degrees of homogeneity in $(\uz;\ux,\uy)$.  These polynomials carry the regular representation of the orthogonal group (or the derived $\mathfrak{so}(m)$-action in terms of angular momentum operators $L_{ab}$ from above). \par
We further introduce the Weyl algebra in three vector variables as the algebra generated by the variables and their corresponding derivatives:
\[
\mathcal{W}(\mathbb{R}^{3m},\mathbb{C}):=\mathsf{Alg}(x_{\alpha},y_{\beta},z_{\gamma},\partial_{x_{\delta}}, \partial_{y_{\epsilon}},\partial_{z_{\zeta}}) \ \ \mbox{with}\ \alpha,\beta,\gamma,\delta,\epsilon,\zeta\in\{1,\dots,m\}\ .
\]
Just like in the case of the classical Fischer decomposition, where the Lie algebra $\mathfrak{sl}(2)$ appears as a Howe dual partner, there is a Lie algebra appearing here. To be precise, it is the Lie algebra $\mathfrak{sp}(6)= \mathfrak{g}_{-2}\oplus \mathfrak{g}_0\oplus \mathfrak{g}_{+2}$, 
with parabolic subalgebra $\mathfrak{p}:=\mathfrak{g}_{-2}\oplus\mathfrak{g}_0$ and Levi subalgebra $\mathfrak{g}_0\cong \mathfrak{gl}(3)$. The subspaces $\mathfrak{g}_{\pm 2}$ contain six `pure' operators each (i.e.\ only variables, acting as a multiplication operator, or only derivatives). 
More specifically, the subspaces are spanned by the following $\mathsf{SO}(m)$-invariant operators:
\begin{align*}
\mathfrak{g}_{-2}:=&\spa(\Delta_x, \Delta_y, \Delta_z, \langle \up_x,\up_y\rangle, \langle \up_y,\up_z\rangle, \langle \up_x,\up_z\rangle)
\\ \mathfrak{g}_0:=&\spa(\langle \ux,\up_y\rangle, \langle \uy,\up_x\rangle, \langle \ux,\up_z \rangle, \langle \uz,\up_x\rangle, \langle \uy,\up_z\rangle, \langle \uz,\up_y\rangle,\mathbb{E}_x,\mathbb{E}_y, \mathbb{E}_z   )
\\ \mathfrak{g}_{+2}:=&\spa(|\ux|^2,|\uy|^2,|\uz|^2, \langle \ux , \uy \rangle , \langle \uy , \uz \rangle, \langle \ux , \uz \rangle )
\end{align*}
\begin{defn}
	The space of \textit{Howe harmonics} of degree $(a,b,c)$ in the variables $(\uz,\ux,\uy)$ is defined as $\mathcal{H}_{a,b,c}^*(\mathbb{R}^{3m},\mathbb{C}) :=\mathcal{P}_{a,b,c}(\mathbb{R}^{3m},\mathbb{C}) \cap \ker(\mathfrak{g}_{-2})$. 
\end{defn}
\noindent In what follows the notation $\ker(A_1,\ldots,A_n)$ stands for $\ker(A_1) \cap \ldots \cap \ker(A_n)$, so $\ker(\mathfrak{g}_{-2})$ means that simplicial harmonics are annihilated by all (pure differential) operators in $\mathfrak{sp}(6)$.
As a representation space for $\mathfrak{so}(m)$, the spaces $\mcH_{a,b,c}^*$ are {\em not} irreducible. In order to obtain an irreducible (sub)space, we have to impose extra conditions. 
\begin{defn}	\label{def_SH}
	The vector space of \textit{simplicial harmonics} of degree $(a,b,c)$ in the variables $(\uz,\ux,\uy)$ is defined by means of
	\[	\mathcal{H}_{a,b,c}(\mathbb{R}^{3m},\mathbb{C}) :=\mathcal{H}^*_{a,b,c}(\mathbb{R}^{3m},\mathbb{C}) \cap \ker\left(\langle \uz,\up_x\rangle, \langle \uz,\up_y\rangle, \langle \ux,\up_y\rangle\right)\ .	\]
\end{defn}
\noindent As was shown in \cite{VSC}, this defines an irreducible representation space for $\mathfrak{so}(m)$ with highest weight $(a,b,c)$, where the dominant weight condition $a\geq b\geq c$ must hold. This now leads to the following generalisation of the result above (the Fisher decompostion in three vector variables):
\begin{thm}
	The space $\mathcal{P}(\R^{3m},\C)$ of complex-valued polynomials in three vector variables (in $\mathbb{R}^m$) has a multiplicity-free decomposition under the action of $\mathfrak{sp}(6)\times\mathsf{SO}(m)$ by means of:
	\begin{align*}
	\mathcal{P}(\R^{3m},\C) \cong \bigoplus_{a\geq b\geq c} \mathbb{V}_{a,b,c}^{\infty}\otimes \mathcal{H}_{a,b,c}(\R^{3m},\C),
	\end{align*}
	where we used the dominant weight condition in the summation. The notation $\mathbb{V}_{a,b,c}^{\infty}$ hereby refers to a Verma module (see for example \cite{Howe}) for $\mathfrak{sp}(6)$.
\end{thm}

\section{The Mickelsson-Zhelobenko algebra (general setup)}
\noindent We have now introduced 21 differential operators giving rise to a realisation of the Lie algebra $\mathfrak{sp}(6)$ inside the Weyl algebra (on 3 vector variables in $\mathbb{R}^m$). In this section we construct a related algebra, the so-called Mickelsson-Zhelobenko algebra (also called  transvector or step algebra) $\mathcal{Z}$. Let $\mathfrak{g}$ be a Lie algebra and let $\mathfrak{s}\subset\mathfrak{g}$ be a reductive subalgebra. We then have the decomposition $\mathfrak{g}=\mathfrak{s}\oplus\mathfrak{t}$, where $\mathfrak{t}$ carries an $\mathfrak{s}$-action for the commutator (i.e. $[\mathfrak{s},\mathfrak{t}]\subset\mathfrak{t}$). For $\mathfrak{s}$ we then fix a triangular decomposition $\mathfrak{s}=\mathfrak{s}^-\oplus \mathfrak{h}\oplus \mathfrak{s}^+$, where $\mathfrak{s}^{\pm}$ consists of the positive (resp.\ negative roots) with respect to the Cartan subalgebra $\mathfrak{h} \subset \mathfrak{s}$. We then also define a left ideal $J \subset \mathcal{U}(\mathfrak{g})$ in the universal enveloping algebra $\mathcal{U}(\mathfrak{g})$ by means of $\mathcal{U}(\mathfrak{g})\mathfrak{s}^+$. This allows us to define a certain subalgebra of $\mathcal{U}(\mathfrak{g})$ which is known as the normaliser: $$\operatorname{Norm}(J):=\{u\in \mathcal{U}(\mathfrak{g})\mid Ju\subset J\}.$$ The crucial point is that $J$ is a two-sided ideal of $\operatorname{Norm}(J)$, which allows us two define the quotient algebra $\mathcal{S}(\mathfrak{g},\mathfrak{s})=\operatorname{Norm}(J)/J$ which is known as the \textit{Mickelsson algebra}.\par In a last step of the construction, we consider an extension of $\mathcal{U}(\mathfrak{g})$ to a suitable localisation $\mathcal{U}'(\mathfrak{g})$ given by \[ \mathcal{U}'(\mathfrak{g})= \mathcal{U}'(\mathfrak{g})\otimes_{\mathcal{U}(\mathfrak{h})}\operatorname{Frac}(\mathcal{U}(\mathfrak{h}))\ ,  \]
where $\operatorname{Frac}(\mathcal{U}(\mathfrak{h}))$ is the field of fractions in the (universal enveloping algebra of the) Cartan algebra. The ideal $J'$ can be introduced for this extension too (in a completely similar way) and the corresponding quotient algebra $\mathcal{Z}(\mathfrak{g},\mathfrak{s}):=\operatorname{Norm}(J')/J'$ is the \textit{Mickelsson-Zhelobenko algebra}. These two algebras are naturally identified, since one has that \[ \mathcal{Z}(\mathfrak{g},\mathfrak{s})= \mathcal{S}(\mathfrak{g},\mathfrak{s})\otimes_{\mathcal{U}(\mathfrak{h})}\operatorname{Frac}(\mathcal{U}(\mathfrak{h}))\ . \]
Note that this algebra is sometimes referred to as a `transvector algebra', which is what we will often use in what follows. 
\section{The Mickelsson-Zhelobenko algebra $\mathcal{Z}(\mathfrak{sp}(6),\mathfrak{so}(4))$} 
We will now define a specific example of the construction from above, which will help us to understand how the branching of $\mathbb{S}^\infty_k$ works. First of all, we note the following: 
\begin{lem}\label{commuting} The three (orthogonally invariant) operators
	\[ 
	L := \langle \ux,\up_y \rangle - \frac{1}{2}\Delta_z \qquad 
	R := \langle \uy,\up_x \rangle + \frac{1}{2}|\uz|^2 \qquad
	\mathcal{E} := \mathbb{E}_y - \mathbb{E}_x + \mathbb{E}_z + \frac{n}{2}
	\]
	give rise to yet another copy of the Lie algebra $\mathfrak{sl}(2)$. This Lie algebra commutes with the Lie algebra $\mathfrak{sl}(2)\cong\operatorname{Alg}(D_s,X_s)$. 
\end{lem}
\noindent 
This thus means that we have now obtained a specific realisation for the Lie algebra $\mathfrak{so}(4) \cong \operatorname{Alg}(D_s,X_s)\oplus \operatorname{Alg}(L,R)\cong \mathfrak{sl}(2)\oplus \mathfrak{sl}(2)$ which appears as a subalgebra of $\mathfrak{sp}(6)$. This algebra will play the role of $\mathfrak{s}$ from Section 4. Let us therefore consider the lowest weight vectors in $\mathfrak{so}(4)$:
\[ Y_1 = D_s = \langle \uz, \up_y \rangle - \langle \up_z,\up_x\rangle \quad \mbox{and} \quad Y_2 = L = \langle \ux, \up_y \rangle - \frac{1}{2}\Delta_z\ .\]
We will focus on the solutions of both lowest weight vectors, i.e.\ $\ker(D_s,L)$. Note that the operators in $\mathfrak{sp}(6)$ do not necessarily act as endomorphisms on this space, but the transvector framework allows us to `replace' these operators by (related) transvector algebra generators which do act as endomorphisms. We start with proving the reductiveness of the algebra $\mathfrak{so}(4)$ in $\mathfrak{sp}(6)$. 
\begin{lem}
	The Lie algebra $\mathfrak{so}(4)$ is reductive in $\mathfrak{sp}(6)$.
\end{lem}
\begin{proof}
	We need to show that $\mathfrak{sp}(6)$ decomposes as $\mathfrak{so}(4)+\mathfrak{t}$, where the subspace $\mathfrak{t}$ carries an action of $\mathfrak{so}(4)$. For that purpose we introduce the following 15 (linearly independent) differential operators:
	\begin{align*}
	\Delta_x	&& \langle \uz,\up_x\rangle && \langle \uy,\up_x\rangle -|\uz|^2 && \langle \uy,\uz\rangle && |\uy|^2\\ 
	\langle \up_x,\up_y\rangle 	&& \langle \uz,\up_y\rangle + \langle \up_z,\up_x\rangle && \mE_x-\mE_y+2\mE_z+m && \langle \ux,\uz\rangle - \langle \uy,\up_z\rangle && \langle \ux,\uy\rangle \\ 
	\Delta_y && \langle \up_y,\up_z\rangle  && \langle \ux,\up_y\rangle + \Delta_z && \langle \ux,\up_z\rangle && |\ux|^2
	\end{align*}
	It is now a straightforward computation to check that for each of these operators the commutator with one of the operators in $\mathfrak{so}(4)$ is again a linear combination of the operators above. 
\end{proof}
\noindent In order to construct the generators for the algebra $\mathcal{Z}(\mathfrak{g},\mathfrak{s})$ with $\mathfrak{g} = \mathfrak{sp}(6)$ and $\mathfrak{s} = \mathfrak{so}(4)$, we need the following: 
\begin{defn}
	The \textit{extremal projector} for the Lie algebra $\mathfrak{sl}(2)=\operatorname{Alg}(X,Y,H)$ is the idempotent operator $\pi$ given by the (formal) expression 
	\begin{align}
	\pi : = 1 + \sum_{j=1}^{\infty}\frac{(-1)^j}{j!} \frac{\Gamma(H+2)}{\Gamma(H+2+j)}Y^jX^j\ .
	\end{align}
	This operator satisfies $X\pi = \pi Y = 0$ and $\pi^2 = \pi$. 
\end{defn}
\noindent Note that this operator is defined on the extension $\mathcal{U}'(\mathfrak{sl}(2))$ of the universal enveloping algebra defined earlier, so that formal series containing the operator $H$ in the denominator are well-defined (in practice it will always reduce to a {\em finite} summation). 
\begin{lem}
	The extremal projector $\pi_{\mathfrak{so}(4)}$ is given by the product of the extremal projectors for the Lie algebras $\mathfrak{sl}(2)$, i.e. $
	\pi_{\mathfrak{so}(4)} = \pi_{D_s}\pi_{L}= \pi_{L}\pi_{D_s}$ (the operator appearing as an index here refers to the realisation for $\mathfrak{sl}(2)$ that was used). 
\end{lem}
\begin{proof}
	This is due to the fact that the two copies of $\mathfrak{sl}(2)$ commute.
\end{proof}
\noindent The operator $\pi_{\mathfrak{so}(4)}$ is thus explicitly given by
\begin{align*}
\left( 1 + \sum_{j=1}^{\infty}\frac{(-1)^j}{j!} \frac{\Gamma(\mathbb{E}+2)}{\Gamma(\mathbb{E}+2+j)}X_s^jD_s^j\right)\left( 1 + \sum_{j=1}^{\infty}\frac{(-1)^j}{j!} \frac{\Gamma(\mathcal{E}+2)}{\Gamma(\mathcal{E}+2+j)}R^jL^j\right)
\end{align*}
and satisfies $D_s\pi_{\mathfrak{so}(4)}=L\pi_{\mathfrak{so}(4)}=0=\pi_{\mathfrak{so}(4)}X_s=\pi_{\mathfrak{so}(4)}R$. This means that we now have a natural object that can be used to project polynomials on the intersection of the kernel of the operators $D_s$ and $L$. \par The 15 operators in $\mathfrak{t} \subset \mathfrak{sp}(6)$ as such do not preserve this kernel space (as these operators do not necessarily commute with $D_s$ and $L$), but their projections will belong to $\operatorname{End}(\ker(D_s,L))$. In what follows we will use the notation $Q_{a,b}$, where $a\in\{\pm 2,0\}$ and $b\in\{\pm 4,\pm 2, 0\}$, to denote the operators in $\mathfrak{t}$ (see Lemma 5.2, and the scheme below). For each operator $Q_{a,b}$ we then also define an associated operator $\mathbb{P}_{a,b}:=\pi_{\mathfrak{so}(4)} Q_{a,b}$.  For instance $\mathbb{P}_{4,-2} = \pi_{\mathfrak{so}(4)}|\uy|^2$. \par The $\mathbb{P}$-operators will then be used to define the generators for our transvector algebra. The diagram below should then be seen as the analogue of the 15 operators $Q_{a,b}$ given above, grouped into a $5 \times 3$ rectangle, where each operator $\alpha \in \mathfrak{t}$ carries a label. The meaning of the labels $(a,b)$ comes from the observation that $\mathfrak{t} \cong \mathbb{V}_4 \otimes \mathbb{V}_2$ as a representation for $\mathfrak{sl}(2) \oplus \mathfrak{sl}(2)$, with $\mathbb{V}_n$ the standard notation for the irreducible representation of dimension $(n+1)$. Given an operator $\alpha \in \mathfrak{t}$, the numbers $a$ and $b$ can thus be retrieved as eigenvalues for the commutator action of the Cartan elements in $\mathfrak{so}(4)$. Note that the projection operator $\mathfrak{so}(4)$ commutes with these Cartan elements (i.e. the operators $Q_{a,b}$ and $\mathbb{P}_{a,b}$ indeed carry the same labels).
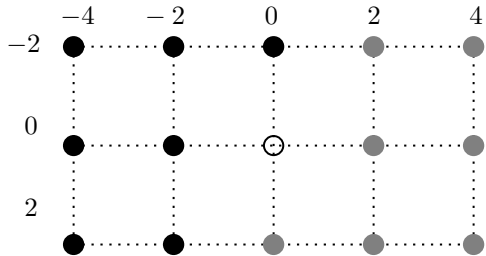
\tikzset{every picture/.style={line width=0.75pt}} 
$$
\begin{tikzpicture}[x=0.75pt,y=0.75pt,yscale=-1,xscale=1]

\draw  [draw opacity=0][dash pattern={on 0.84pt off 2.51pt}] (70.67,32) -- (272.33,32) -- (272.33,133) -- (70.67,133) -- cycle ; \draw  [dash pattern={on 0.84pt off 2.51pt}] (70.67,32) -- (70.67,133)(120.67,32) -- (120.67,133)(170.67,32) -- (170.67,133)(220.67,32) -- (220.67,133)(270.67,32) -- (270.67,133) ; \draw  [dash pattern={on 0.84pt off 2.51pt}] (70.67,32) -- (272.33,32)(70.67,82) -- (272.33,82)(70.67,132) -- (272.33,132) ; \draw  [dash pattern={on 0.84pt off 2.51pt}]  ;
\draw  [fill={rgb, 255:red, 0; green, 0; blue, 0 }  ,fill opacity=1 ] (65.83,32) .. controls (65.83,29.33) and (68,27.17) .. (70.67,27.17) .. controls (73.34,27.17) and (75.5,29.33) .. (75.5,32) .. controls (75.5,34.67) and (73.34,36.83) .. (70.67,36.83) .. controls (68,36.83) and (65.83,34.67) .. (65.83,32) -- cycle ;
\draw  [fill={rgb, 255:red, 0; green, 0; blue, 0 }  ,fill opacity=1 ] (115.83,32) .. controls (115.83,29.33) and (118,27.17) .. (120.67,27.17) .. controls (123.34,27.17) and (125.5,29.33) .. (125.5,32) .. controls (125.5,34.67) and (123.34,36.83) .. (120.67,36.83) .. controls (118,36.83) and (115.83,34.67) .. (115.83,32) -- cycle ;
\draw  [fill={rgb, 255:red, 0; green, 0; blue, 0 }  ,fill opacity=1 ] (165.83,32) .. controls (165.83,29.33) and (168,27.17) .. (170.67,27.17) .. controls (173.34,27.17) and (175.5,29.33) .. (175.5,32) .. controls (175.5,34.67) and (173.34,36.83) .. (170.67,36.83) .. controls (168,36.83) and (165.83,34.67) .. (165.83,32) -- cycle ;
\draw  [fill={rgb, 255:red, 0; green, 0; blue, 0 }  ,fill opacity=1 ] (65.83,82) .. controls (65.83,79.33) and (68,77.17) .. (70.67,77.17) .. controls (73.34,77.17) and (75.5,79.33) .. (75.5,82) .. controls (75.5,84.67) and (73.34,86.83) .. (70.67,86.83) .. controls (68,86.83) and (65.83,84.67) .. (65.83,82) -- cycle ;
\draw  [fill={rgb, 255:red, 0; green, 0; blue, 0 }  ,fill opacity=1 ] (115.83,82) .. controls (115.83,79.33) and (118,77.17) .. (120.67,77.17) .. controls (123.34,77.17) and (125.5,79.33) .. (125.5,82) .. controls (125.5,84.67) and (123.34,86.83) .. (120.67,86.83) .. controls (118,86.83) and (115.83,84.67) .. (115.83,82) -- cycle ;
\draw  [fill={rgb, 255:red, 0; green, 0; blue, 0 }  ,fill opacity=1 ] (65.83,132) .. controls (65.83,129.33) and (68,127.17) .. (70.67,127.17) .. controls (73.34,127.17) and (75.5,129.33) .. (75.5,132) .. controls (75.5,134.67) and (73.34,136.83) .. (70.67,136.83) .. controls (68,136.83) and (65.83,134.67) .. (65.83,132) -- cycle ;
\draw  [fill={rgb, 255:red, 0; green, 0; blue, 0 }  ,fill opacity=1 ] (115.83,132) .. controls (115.83,129.33) and (118,127.17) .. (120.67,127.17) .. controls (123.34,127.17) and (125.5,129.33) .. (125.5,132) .. controls (125.5,134.67) and (123.34,136.83) .. (120.67,136.83) .. controls (118,136.83) and (115.83,134.67) .. (115.83,132) -- cycle ;
\draw  [color={rgb, 255:red, 128; green, 128; blue, 128 }  ,draw opacity=1 ][fill={rgb, 255:red, 128; green, 128; blue, 128 }  ,fill opacity=1 ] (215.83,32) .. controls (215.83,29.33) and (218,27.17) .. (220.67,27.17) .. controls (223.34,27.17) and (225.5,29.33) .. (225.5,32) .. controls (225.5,34.67) and (223.34,36.83) .. (220.67,36.83) .. controls (218,36.83) and (215.83,34.67) .. (215.83,32) -- cycle ;
\draw  [color={rgb, 255:red, 128; green, 128; blue, 128 }  ,draw opacity=1 ][fill={rgb, 255:red, 128; green, 128; blue, 128 }  ,fill opacity=1 ] (265.83,32) .. controls (265.83,29.33) and (268,27.17) .. (270.67,27.17) .. controls (273.34,27.17) and (275.5,29.33) .. (275.5,32) .. controls (275.5,34.67) and (273.34,36.83) .. (270.67,36.83) .. controls (268,36.83) and (265.83,34.67) .. (265.83,32) -- cycle ;
\draw  [color={rgb, 255:red, 128; green, 128; blue, 128 }  ,draw opacity=1 ][fill={rgb, 255:red, 128; green, 128; blue, 128 }  ,fill opacity=1 ] (215.83,82) .. controls (215.83,79.33) and (218,77.17) .. (220.67,77.17) .. controls (223.34,77.17) and (225.5,79.33) .. (225.5,82) .. controls (225.5,84.67) and (223.34,86.83) .. (220.67,86.83) .. controls (218,86.83) and (215.83,84.67) .. (215.83,82) -- cycle ;
\draw  [color={rgb, 255:red, 128; green, 128; blue, 128 }  ,draw opacity=1 ][fill={rgb, 255:red, 128; green, 128; blue, 128 }  ,fill opacity=1 ] (265.83,82) .. controls (265.83,79.33) and (268,77.17) .. (270.67,77.17) .. controls (273.34,77.17) and (275.5,79.33) .. (275.5,82) .. controls (275.5,84.67) and (273.34,86.83) .. (270.67,86.83) .. controls (268,86.83) and (265.83,84.67) .. (265.83,82) -- cycle ;
\draw  [color={rgb, 255:red, 128; green, 128; blue, 128 }  ,draw opacity=1 ][fill={rgb, 255:red, 128; green, 128; blue, 128 }  ,fill opacity=1 ] (265.83,132) .. controls (265.83,129.33) and (268,127.17) .. (270.67,127.17) .. controls (273.34,127.17) and (275.5,129.33) .. (275.5,132) .. controls (275.5,134.67) and (273.34,136.83) .. (270.67,136.83) .. controls (268,136.83) and (265.83,134.67) .. (265.83,132) -- cycle ;
\draw  [color={rgb, 255:red, 128; green, 128; blue, 128 }  ,draw opacity=1 ][fill={rgb, 255:red, 128; green, 128; blue, 128 }  ,fill opacity=1 ] (215.83,132) .. controls (215.83,129.33) and (218,127.17) .. (220.67,127.17) .. controls (223.34,127.17) and (225.5,129.33) .. (225.5,132) .. controls (225.5,134.67) and (223.34,136.83) .. (220.67,136.83) .. controls (218,136.83) and (215.83,134.67) .. (215.83,132) -- cycle ;
\draw  [color={rgb, 255:red, 128; green, 128; blue, 128 }  ,draw opacity=1 ][fill={rgb, 255:red, 128; green, 128; blue, 128 }  ,fill opacity=1 ] (165.83,132) .. controls (165.83,129.33) and (168,127.17) .. (170.67,127.17) .. controls (173.34,127.17) and (175.5,129.33) .. (175.5,132) .. controls (175.5,134.67) and (173.34,136.83) .. (170.67,136.83) .. controls (168,136.83) and (165.83,134.67) .. (165.83,132) -- cycle ;
\draw  [color={rgb, 255:red, 0; green, 0; blue, 0 }  ,draw opacity=1 ] (165.83,82) .. controls (165.83,79.33) and (168,77.17) .. (170.67,77.17) .. controls (173.34,77.17) and (175.5,79.33) .. (175.5,82) .. controls (175.5,84.67) and (173.34,86.83) .. (170.67,86.83) .. controls (168,86.83) and (165.83,84.67) .. (165.83,82) -- cycle ;

\draw (29,22) node [anchor=north west][inner sep=0.75pt]    {$ \begin{array}{l}
	-2\\ \vspace{0,25cm}
	\\ \vspace{0,25cm}
	\ \ 0\\
	\\ 
	\ \ 2
	\end{array}$};
\draw (62.67,10.4) node [anchor=north west][inner sep=0.75pt]    {$-4\ \ \ \ \ -2\ \ \ \ \ \ \ \ \ 0\ \ \ \ \ \ \ \ \ \  2\ \ \ \ \ \ \ \ \ \ 4$};

\end{tikzpicture}
$$
\par \noindent 
Despite the fact that $\mathcal{Z}(\mathfrak{sp}(6),\mathfrak{so}(4))$ is \textit{not} a Lie algebra, we have organised these operators in such a way that the notions of `positive' and `negative' roots can be used. To be more precise: black dots (resp. grey dots) refer to negative (resp. positive) operators, and the white dot plays the role of a `Cartan element' (this analogy will come in handy below). The 7 black dots (resp. 7 grey dots) will be referred to as operators in $\rho^-$ (resp. in $\rho^+$). Together with the operator $\mathbb{P}_{0,0}$ we then get the set $$\mathcal{G}_{\mathcal{Z}} = \{\mathbb{P}_{a,b} : a \in \{\pm 2, 0\}, b \in \{\pm 4, \pm 2, 0\}\},$$ containing all the generators for the transvector algebra $\mathcal{Z}(\mathfrak{sp}(6),\mathfrak{so}(4))$.\par  Due to a general result by Zhelobenko, these generators then satisfy {\em quadratic} relations (i.e.\ different from the classical Lie brackets). 
 In the next theorem, we will relate the spaces $\mathcal{H}_{a,b,c}(\mathbb{R}^{3m},\mathbb{C})$ introduced in Definition \ref{def_SH} to the space of polynomial solutions for the symplectic Dirac operator $D_s$, the lowering operator $L$ and the negative `roots' $\rho^-$ which we have just introduced (i.e.\ the operators $\mathbb{P}_{a,b}$ corresponding to black dots).

\begin{thm} 
The solutions for the operators $D_s$ and $L$ and the negative roots $\rho^- \subset \mathcal{G}_{\mathcal{Z}}$ which are homogeneous of degree $(a,b,c)$ in the variables $(\uz,\ux,\uy)$ are precisely given by the simplicial harmonics $\mathcal{H}_{a,b,c}(\mathbb{R}^{3m},\mathbb{C})$. In other words, we have:
\begin{align*}
\mcP_{a,b,c}(\mR^{3m},\mC) \cap \ker(D_s, L,\rho^-) = \mathcal{H}_{a,b,c}(\mathbb{R}^{3m},\mathbb{C}).
\end{align*}
\end{thm}

\begin{proof}
The idea behind this proof is a recursive argument, where the ordering on the black dots will be from left to right and from bottom to top in the rectangular scheme above (in terms of labels this means that $(2,-4) > (0,-4) > (2,-2)$, as an example). The reason for doing so is the following: the commutators $[L,Q_{a,b}]$ and $[D_s,Q_{a,b}]$ give an operator situated below or to the left of the operator $Q_{a,b}$ we started from. Up to a constant, these operators are equal to $Q_{a+2,b}$ and $Q_{a,b-2}$ respectively (or trivial whenever the parameters $a$ and $b$ are not in the correct range). This means that combinations of the form $LQ_{a,b}$ and $D_sQ_{ab}$ act trivially on functions $H(\uz;\ux,\uy)$ in the kernel of $L$ and $D_s$, provided we know that also $Q_{a+2,b}$ and $Q_{a,b-2}$ act trivially. Given the fact that each operator $\mathbb{P}_{a,b} \in \rho^-$ is of the form 
\[ \mathbb{P}_{a,b} = \big(1 + \mathcal{O}_1L\big)\big(1 + \mathcal{O}_2D_s\big)Q_{a,b}\ , \]
where $\mathcal{O}_j$ is a short-hand notation for the correction terms coming from the extremal projection operator (which, unless this operator reduces to the identity operator, always contains either an operator $L$ or $D_s$ at the right). The upshot of our recursive scheme is that once we know that $Q_{a+2,b}$ and $Q_{a,b-2}$ act trivially, this immediately tells us that $\mathbb{P}_{a,b} H = 0 \Rightarrow Q_{a,b} H = 0$. Because $\mathbb{P}_{2,-4}H = 0$ and $\mathbb{P}_{2,-4} = Q_{-2,4} = \Delta_y$, we can immediately conclude that the following operators will then act trivially: 
\[ \Delta_y \quad \langle \up_x,\up_y \rangle \quad \Delta_x \quad \langle \up_y,\up_z \rangle \quad \langle \uz,\up_y \rangle + \langle \up_x,\up_z \rangle \quad \langle \uz,\up_x \rangle \quad \langle \ux,\up_y \rangle + \Delta_z\ . \]
In order to be simplicial harmonic, $H(\uz;\ux,\uy)$ should belong to the kernel of 9 operators in $\mathfrak{sp}(6)$ (see Definition \ref{def_SH}), but it is straightforward to see that one can reproduce these operators as commutators of the 7 operators on the previous line. For example: $\Delta_x(\langle \ux,\up_y \rangle + \Delta_z)H = 0$ leads to $\Delta_z H = 0$, since $\langle \up_x,\up_y \rangle H = 0$ (and so on).  
\end{proof}

\section{Application: branching symplectic monogenics}
We will now use the operators $\mathbb{P}_{a,b}$ to explicitly describe the branching of the $k$-homogeneous symplectic monogenics $\mathbb{S}_k^{\infty}$. By this we mean that it will give us a systematic way to define the `embedding factors' realising the isomorphic copy of those spaces in $\mathbb{S}_k^{\infty}$. To do so, we will make an analogy again: one can consider the asssociative algebra $\mathcal{U}(\mathcal{Z})$, the `universal enveloping algebra' of $\mathcal{Z}(\mathfrak{sp}(6),\mathfrak{so}(4))$. The meaning should be clear here: it is a tensor algebra $\bigotimes V$ (with $V$ the span of $\mathcal{G}_{\mathcal{Z}}$-generators as an underlying vector space) modulo the ideal spanned `by the quadratic relations' in the transvector algebra. We will refer to elements in this algebra as `words' in `an alphabet' that can be ordered. This statement, which should thus be seen as an analogue of the Poincar\'e--Birkhoff--Witt theorem (PBW theorem), requires a proof but we will not do this in the present paper. As a matter of fact, the general case $k \in \mZ^+$ will be treated in an upcoming (longer) paper, in the present article we will focus on the case $k = 1$ as a guiding example. \par The main idea is the following: imposing the lexicographic ordering on the labels $(a,b)$ will dictate the position of our letters in the alphabet (from left to right), with e.g.\ $(4,0) > (4,-2) > (2,2)$. Letting such a word acting as an operator on simplicial harmonics $H_{a,b,c}(\uz;\ux,\uy)$, it should be clear (in view of the previous theorem) that only the `letters' corresponding to grey dots in the scheme will play a role (the white dot acts as a constant, whereas the black dots act trivially). Considering the fact that the total degree of `a word' in $\ux$ and $\uy$ should not exceed $k=1$, we can only use the operators $\mathbb{P}_{a,b}$ from the third and fourth column in our example. Note that once the operator $\mathbb{P}_{ab}$ has been chosen (i.e. \ the `word' in front of the simplicial harmonics), the degree $(a,b,c)$ of these polynomials $H_{a,b,c}(\uz;\ux,\uy)$ is automatically fixed too: the total degree in $\uz$ and $(\ux,\uy)$ is then equal to $k$ and $1$ respectively. So, when the `word' is homogeneous of degree one in $(\ux,\uy)$ we get contributions of the form $\mathbb{P}_{0,0}\mathcal{H}_{a,1,0}$ and $\mathbb{P}_{2,0}\mathcal{H}_{a,1,0}$.
Whereas when the chosen `word' is homogeneous of degree zero we get $
	\mathbb{P}_{-2,2}\mathcal{H}_{a,0,0}$, $\mathbb{P}_{0,2}\mathcal{H}_{a,0,0}$ and $\mathbb{P}_{2,2}\mathcal{H}_{a,0,0}$. 
Finally, we note that we can still act with the raising operator $R \in \mathfrak{sl}(2)$ on each of the polynomials from above (i.e.\ a suitable projection operator acting on a suitable space of simplicial harmonics) to arrive at a direct sum of Verma modules which can be embedded into $\mathbb{S}^\infty_1$. This is based on the trivial albeit crucial observation that $[R,D_s] = 0$, so that acting with $R$ preserves symplectic monogenic solutions. This means that we have now resolved the branching problem for $k=1$ in a completely different way. Resulting in the decomposition

\begin{align*}
\mathbb{S}^{\infty}_1\bigg\downarrow^{\mathfrak{sp}(2m)}_{\mathfrak{so}(m)} &\cong \bigoplus_{a \geq 1}\bigoplus_{\ell = 0}^\infty R^\ell (\mathcal{H}_{a,1}\oplus \mathbb{P}_{2,0} \mathcal{H}_{a,1})\nonumber\\ 
& \oplus \bigoplus_{a \geq 0}\bigoplus_{\ell = 0}^\infty R^\ell (\mathbb{P}_{-2,2}\mathcal{H}_a\oplus \mathbb{P}_{-2,0}\mathcal{H}_a \oplus \mathbb{P}_{-2,-2}\mathcal{H}_a).
\end{align*}
Summarising the idea behind this decomposition, we thus claim that $\mS^\infty_k$ can be decomposed under the joint action of $$\mathfrak{so}(m)\times \mathfrak{sl}(2)\times \mathcal{Z}(\mathfrak{sp}(6),\mathfrak{so}(4)),$$ whereby the final decomposition will contain summands of the form $$R^p\left(\mathcal{U}(\rho^+)\mathcal{H}_{a,b,c}\right)$$ for suitable `words' in the algebra $\mathcal{U}(\mathcal{\rho^+})$ and suitable spaces of simplicial harmonics.

\subsection*{Acknowledgments}
The author G.M. was supported by the FWO-EoS project G0H4518N.

\end{document}